\documentclass[reqno]{amsart}
\usepackage{amssymb}
\usepackage{graphicx}

\usepackage[usenames, dvipsnames]{color}
\usepackage{verbatim}
\usepackage{mathrsfs}
\usepackage{bm}
\usepackage{cite}
\usepackage{esint}

\numberwithin{equation}{section}

\newtheorem{theorem}{Theorem}[section]

\theoremstyle{definition}
\newtheorem{remark}[theorem]{Remark}

\theoremstyle{definition}

\theoremstyle{definition}

\makeatletter
\def\dashint{\operatorname%
{\,\,\text{\bf-}\kern-.98em\DOTSI\intop\ilimits@\!\!}}
\makeatother

\def\\det{\text{det}}

\def\.5{\frac{1}{2}}

\newcommand{\RN}[1]{%
  \textup{\uppercase\expandafter{\romannumeral#1}}%
}

\renewcommand{\epsilon}{\varepsilon}

\newcounter{marnote}

\begin{document}
\title[Exact solutions for the insulated and perfect conductivity problems]{Exact solutions for the insulated and perfect conductivity problems with concentric balls}


\author[Z.W. Zhao]{Zhiwen Zhao}

\address[Z.W. Zhao]{Beijing Computational Science Research Center, Beijing 100193, China.}
\email{zwzhao365@163.com}

\date{\today} 



\begin{abstract}
The insulated and perfect conductivity problems arising from high-contrast composite materials are considered in all dimensions. The solution and its gradient, respectively, represent the electric potential and field. The novelty of this paper lies in finding exact solutions for the insulated and perfect conductivity problems with concentric balls. Our results show that there appears no electric field concentration for the insulated conductivity problem, while the electric field for the perfect conductivity problem exhibits sharp singularity with respect to the small distance between interfacial boundaries of the interior and exterior balls. This discrepancy reveals that concentric balls is the optimal structure of insulated composites, but not for superconducting composites.

\end{abstract}

\maketitle



\section{Introduction and main results}

High-contrast fiber-reinforced composite materials have been widely used in industrial field due to their properties of less expensive, lighter, stronger or more durable by contrast with common materials. Each year, high-contrast composites bring hundreds of new technical innovations and applications, from golf clubs, tennis rackets, buildings, bridges to microchips, aircraft, missiles and spacecraft. Quantitative analysis on mechanical properties of high-contrast composites is the key to their applications, especially to the finding of optimal structure of composites.

High-contrast composites with closely located inclusions can be described by the insulated and perfect conductivity equations. It is well known that the electric field, which is the gradient of solution, concentrates highly in the thin gaps between inclusions. The problem of estimating the gradient of solution has been actively investigated since the famous work of Babu\u{s}ka et al. \cite{BASL1999} and has resulted in a large number of papers involving different methods and cases based on shapes of inclusions, dimensions, applied boundary conditions, for example, see \cite{W2021,Y2016,DLY2021,DLY2022,ACKLY2013,AKL2005,AKLLL2007,BLY2009,BLY2010,BT2013,LY202102,LY2009,Y2007,Y2009} and the references therein. The gradient blow-up rate has been captured in these papers. To be specific, with regard to the perfect conductivity problem, when the distance $\varepsilon$ between two strictly convex inclusions goes to zero, the gradient of the solution blows up at the rate of $\varepsilon^{-1/2}$, $|\varepsilon\ln\varepsilon|^{-1}$ and $\varepsilon^{-1}$, respectively, in dimensions $d=2$, $d=3$ and $d\geq4$. Different from the perfect conductivity problem, it has been recently proved in \cite{DLY2022} that the gradient blow-up rate for the insulated conductivity problem depends not only on the dimension but also on the principal curvature of the surfaces of inclusions. In addition, for nonlinear equation, we refer to \cite{G2012,CS2019,CS201902}.

We should point out that the volumes of the whole matrix domains considered in all the aforementioned work are of constant order with respect to the distance $\varepsilon$ and the electric field only appears blow-up in the narrow region between two inclusions. So a natural problem is that when the volume of the considered matrix domain degenerates to be of infinitely small quantity order in terms of $\varepsilon$ such as order $O(\varepsilon)$, that is, the whole matrix domain becomes the thin gap, whether the electric field concentrates highly in the whole matrix domain. In order to answer this question, in this paper we consider a mathematical model of high-contrast composite materials with the core-shell geometry modeled by concentric balls in all dimensions. Let $B_{r_{0}},B_{r_{0}+\varepsilon}\subset\mathbb{R}^{d}\,(d\geq2)$ be two balls centered at the origin with the radii $r_{0}$ and $r_{0}+\varepsilon$, respectively, where $\varepsilon$ is a positive constant. In the presence of concentric balls $B_{r_{0}}$ and $B_{r_{0}+\varepsilon}$, we consider the following insulated and perfect conductivity problem:
\begin{equation}\label{ZZWW001}
\begin{cases}
\Delta{u}=0,& \mbox{in}~B_{r_{0}+\varepsilon}\setminus\overline{B}_{r_{0}},\\
\frac{\partial{u}}{\partial\nu}=0,&\mbox{on}~\partial B_{r_{0}},\\
u=\varphi,&\mbox{on}~\partial{B}_{r_{0}+\varepsilon},
\end{cases}
\end{equation}
and
\begin{equation}\label{YHU0.002}
\begin{cases}
\Delta{u}=0,& \mbox{in}~B_{r_{0}+\varepsilon}\setminus\overline{B}_{r_{0}},\\
u=C^{0},&\mbox{on}~\overline{B}_{r_{0}},\\
\int_{\partial{B}_{r_{0}}}\frac{\partial{u}}{\partial\nu}=0,\\
u=\varphi,&\mbox{on}~\partial{B}_{r_{0}+\varepsilon},
\end{cases}
\end{equation}
where $C^{0}$ is the free constant determined by the third line of \eqref{YHU0.002}, $\nu$ denotes the unit outer normal to the domain, the boundary data $\varphi$ belongs to $C^{2}(\partial B_{r_{0}+\varepsilon})\cap C^{k,\alpha}(\partial B_{r_{0}+\varepsilon})$ with $k\geq0,\,\alpha\in(0,1]$ satisfying the following condition:
\begin{align}\label{INDEX006}
\text{$k+\alpha>\frac{1}{2}$ if $d=2$, and $k+\alpha>\frac{d}{2}-1$ if $d\geq3$.}
\end{align} 
In physics, the interior ball $B_{r_{0}}$ denotes the stiff inclusion (or the fiber) and the domain $B_{r_{0}+\varepsilon}\setminus\overline{B}_{r_{0}}$ represents the matrix. With regard to the existence, uniqueness and regularity of weak solutions to problems \eqref{ZZWW001} and \eqref{YHU0.002}, see \cite{BLY2009,BLY2010}.

For $x\in \overline{B_{r_{0}+\varepsilon}\setminus B_{r_{0}}}$, write the polar coordinates as $x=(r,\xi)\in[r_{0},r_{0}+\varepsilon]\times\mathbb{S}^{d-1}$, where $r=|x|$. Let $N_{k,d}$ represent the dimension of the space of the homogeneous harmonics of degree $k$ in $d$ dimensions, whose value is given as follows:
\begin{align}\label{DIMEN001}
N_{k,d}=
\begin{cases}
(2k+d-2)\frac{(k+d-3)!}{k!(d-2)!},&k\geq1,\;d\geq2,\\
1,&k=0,\,d\geq2.
\end{cases}
\end{align}
Then using Theorem 2.36 and Corollary 4.14 in \cite{AH2012}, we can use spherical harmonics to expand the above boundary data $\varphi$ as follows:
\begin{align}\label{WZMAQ001}
\varphi((r_{0}+\varepsilon)\xi)=\lim_{m\rightarrow\infty}\sum^{m}_{k=0}\sum^{N_{k,d}}_{l=1}\langle \varphi,Y_{k,l}\rangle_{_{\mathbb{S}^{d-1}}}Y_{k,l}(\xi),\quad\text{uniformly in}\;\xi\in\mathbb{S}^{d-1},
\end{align}
where $N_{k,d}$ is given by \eqref{DIMEN001}, $\langle \varphi,Y_{k,l}\rangle_{_{\mathbb{S}^{d-1}}}=\int_{\mathbb{S}^{d-1}}\varphi((r_{0}+\varepsilon)\xi)Y_{k,l}(\xi)d\xi$, the set $\{Y_{k,l}\}_{k,l}$, whose element $Y_{k,l}(\xi)$ is a spherical harmonic of degree $k$, forms an orthonormal basis of $L^{2}(\mathbb{S}^{d-1})$. Note that for $k=0$, we have $Y_{0,0}(\xi)=|\mathbb{S}^{d-1}|^{-1/2}$. The expansion in \eqref{WZMAQ001} actually provides a selection method of finite-dimensional approximation sequence for the boundary data $\varphi$. For simplicity, for any integer $m\geq0$, denote
\begin{align}\label{AMZW001}
\varphi_{m}(x):=\sum^{m}_{k=0}\sum^{N_{k,d}}_{l=1}a_{kl} Y_{k,l}(\xi),\quad\mathrm{on}\;\partial B_{r_{0}+\varepsilon},
\end{align}
where $a_{kl}:=\langle \varphi,Y_{k,l}\rangle_{_{\mathbb{S}^{d-1}}}=\int_{\mathbb{S}^{d-1}}\varphi((r_{0}+\varepsilon)\xi)Y_{k,l}(\xi)d\xi$. Remark that if $\varphi\not\equiv C$ on $\partial B_{r_{0}+\varepsilon}$, then we have $m>0$. Then we restate \eqref{WZMAQ001} as follows:
\begin{align}\label{EXPAN005}
\sup\limits_{x\in\partial B_{r_{0}+\varepsilon}}|\varphi_{m}(x)-\varphi(x)|\rightarrow0,\quad\mathrm{as}\;m\rightarrow\infty.
\end{align}
For $r\in[r_{0},r_{0}+\varepsilon]$ and any integer $i$, define
\begin{align}\label{NOTA001}
\rho_{i}(r):=
\begin{cases}
-\ln r,&i=2,\\
r^{2-i},&i\neq2.
\end{cases}
\end{align}
For $k\geq1$, introduce some constants as follows:
\begin{align}\label{COEFF001}
\begin{cases}
c^{_{(1)}}_{k}:=\frac{k\rho_{4-d-2k}(r_{0})\rho_{4-d-k}(r_{0}+\varepsilon)}{k\rho_{4-d-2k}(r_{0})+(k+d-2)\rho_{4-d-2k}(r_{0}+\varepsilon)},\vspace{0.5ex}\\
c^{_{(2)}}_{k}:=\frac{(k+d-2)\rho_{4-d-k}(r_{0}+\varepsilon)}{k\rho_{4-d-2k}(r_{0})+(k+d-2)\rho_{4-d-2k}(r_{0}+\varepsilon)}.
\end{cases}
\end{align}
The first main result is concerned with explicit solution for the insulated conductivity problem with concentric balls.
\begin{theorem}\label{thm1}
Suppose that $d\geq2$ and $\varepsilon>0$. For the nonconstant boundary data $\varphi\in C^{2}(\partial B_{r_{0}+\varepsilon})\cap C^{k,\alpha}(\partial B_{r_{0}+\varepsilon})$ with indices $k\geq0,\,\alpha\in(0,1]$ satisfying condition \eqref{INDEX006}, let $u\in H^{1}(B_{r_{0}+\varepsilon})\cap C^{1}(\overline{B_{r_{0}+\varepsilon}\setminus B_{r_{0}}})$ be the solution of the insulated conductivity problem \eqref{ZZWW001}. Then for $x=r\xi\in B_{r_{0}+\varepsilon}\setminus \overline{B}_{r_{0}}$,
\begin{align}\label{RESULT001}
u(x)=a_{00}|\mathbb{S}^{d-1}|^{-1/2}+\sum^{\infty}_{k=1}\sum^{N_{k,d}}_{l=1}a_{kl}(c^{_{(1)}}_{k}r^{2-d-k}+c^{_{(2)}}_{k}r^{k})Y_{k,l}(\xi),
\end{align}
where $N_{k,d}$ is given in \eqref{DIMEN001}, $a_{kl}=\int_{\mathbb{S}^{d-1}}\varphi((r_{0}+\varepsilon)\xi)Y_{k,l}(\xi)d\xi$, the values of $c^{_{(i)}}_{k}$, $i=1,2,\,k\geq1$ are given by \eqref{COEFF001}.

\end{theorem}

\begin{remark}\label{EMA001}
From \eqref{RESULT001}, we obtain that for $x=r\xi\in B_{r_{0}+\varepsilon}\setminus \overline{B}_{r_{0}}$,
\begin{align}\label{GRA001}
\nabla u(x)=&\xi\partial_{r}u+\frac{1}{r}\nabla_{\mathbb{S}^{d-1}}u\notag\\
=&\sum^{\infty}_{k=1}\sum^{N_{k,d}}_{l=1}a_{kl}\big(c^{_{(1)}}_{k}(2-d-k)r^{1-d-k}+c^{_{(2)}}_{k}kr^{k-1}\big)Y_{k,l}(\xi)\xi\notag\\
&+\sum^{\infty}_{k=1}\sum^{N_{k,d}}_{l=1}a_{kl}(c^{_{(1)}}_{k}r^{1-d-k}+c^{_{(2)}}_{k}r^{k-1})\nabla_{\mathbb{S}^{d-1}}Y_{k,l}(\xi),
\end{align}
where $\nabla_{\mathbb{S}^{d-1}}$ is the first-order Beltrami operator on $\mathbb{S}^{d-1}$. Observe that
$$\|\varphi\|_{L^{2}(\partial B_{r_{0}+\varepsilon})}^{2}=\sum^{\infty}_{k=0}\sum^{N_{k,d}}_{l=1}a_{kl}^{2},$$
which, together with the fact that $r\in(r_{0},r_{0}+\varepsilon)$, reads that
$$|\nabla u|\leq C\bigg(\sum^{\infty}_{k=0}\sum^{N_{k,d}}_{l=1}a_{kl}^{2}\bigg)^{\frac{1}{2}}\leq C\|\varphi\|_{L^{\infty}(\partial B_{r_{0}+\varepsilon})},$$
for some $\varepsilon$-independent constant $C$. This implies that concentric balls is the optimal insulation structure of composite materials.
\end{remark}
\begin{remark}
By contrast with all the previous work on the blow-up of the electric field for the insulated and perfect conductivity problems, the novelty of our results obtained in Theorems \ref{thm1} and \ref{thm2} lies in the following two aspects. On one hand, the solution and its gradient are exactly given. On the other hand, our results in Theorems \ref{thm1} and \ref{thm2} hold for any $\varepsilon>0$, that is, the distance $\varepsilon$ allows to be a large positive constant.

\end{remark}

For $k\geq1$ and $r\in[r_{0},r_{0}+\varepsilon]$, define
\begin{align}\label{CONSTANT006}
\begin{cases}
\tilde{c}^{_{(1)}}_{k}(r)=\frac{\rho_{d}(r)-\rho_{d}(r_{0}+\varepsilon)}{\rho_{d}(r_{0})-\rho_{d}(r_{0}+\varepsilon)},\vspace{0.5ex}\\
\tilde{c}^{_{(2)}}_{k}(r)=\frac{(\rho_{2-k}(r)-\rho_{4-d-2k}(r_{0})\rho_{d+k}(r))\rho_{4-d-k}(r_{0}+\varepsilon)}{\rho_{4-d-2k}(r_{0}+\varepsilon)-\rho_{4-d-2k}(r_{0})}.
\end{cases}
\end{align}
With regard to the perfect conductivity problem with concentric balls, we have
\begin{theorem}\label{thm2}
Suppose that $d\geq2$ and $\varepsilon>0$. For the nonconstant boundary data $\varphi\in C^{2}(\partial B_{r_{0}+\varepsilon})\cap C^{k,\alpha}(\partial B_{r_{0}+\varepsilon})$ with indices $k\geq0,\,\alpha\in(0,1]$ satisfying condition \eqref{INDEX006}, let $u\in H^{1}(B_{r_{0}+\varepsilon})\cap C^{1}(\overline{B_{r_{0}+\varepsilon}\setminus B_{r_{0}}})$ be the solution of the perfect conductivity problem \eqref{YHU0.002}. Then for $x=r\xi\in B_{r_{0}+\varepsilon}\setminus \overline{B}_{r_{0}}$,
\begin{align}\label{QATZ002}
u(x)=\frac{a_{00}}{|\mathbb{S}^{d-1}|^{1/2}}+\sum^{\infty}_{k=1}\sum^{N_{k,d}}_{l=1}a_{kl}\Big(\tilde{c}^{_{(1)}}_{k}(r)\fint_{\mathbb{S}^{d-1}}Y_{k,l}(\xi)d\xi+\tilde{c}^{_{(2)}}_{k}(r)Y_{k,l}(\xi)\Big),
\end{align}
where $N_{k,d}$ is given by \eqref{DIMEN001}, $a_{kl}=\int_{\mathbb{S}^{d-1}}\varphi((r_{0}+\varepsilon)\xi)Y_{k,l}(\xi)d\xi$, $\tilde{c}^{_{(i)}}_{k}(r)$, $i=1,2,\,k\geq1$ are defined by \eqref{CONSTANT006}.

\end{theorem}

\begin{remark}
Similar to \eqref{GRA001}, we deduce from \eqref{QATZ002} that for $x=r\xi\in B_{r_{0}+\varepsilon}\setminus \overline{B}_{r_{0}}$,
\begin{align}\label{QMEIZ001}
\nabla u(x)
=&\sum^{\infty}_{k=1}\sum^{N_{k,d}}_{l=1}a_{kl}\Big(\partial_{r}\tilde{c}^{_{(1)}}_{k}(r)\fint_{\mathbb{S}^{d-1}}Y_{k,l}(\xi)d\xi+\partial_{r}\tilde{c}^{_{(2)}}_{k}(r)Y_{k,l}(\xi)\Big)\notag\\
&+\sum^{\infty}_{k=1}\sum^{N_{k,d}}_{l=1}a_{kl}\frac{\tilde{c}^{_{(2)}}_{k}(r)}{r}\nabla_{\mathbb{S}^{d-1}}Y_{k,l}(\xi).
\end{align}
By Taylor expansion, it follows from a direct calculation that for a sufficiently small $\varepsilon>0$,
\begin{align*}
\partial_{r}\rho_{d}(r)=&
\begin{cases}
-r^{-1},&d=2,\\
-(d-2)r^{1-d},&d\geq3,
\end{cases}\quad
\begin{cases}
\partial_{r}\rho_{2-k}(r)=kr^{k-1},\\
\partial_{r}\rho_{d+k}(r)=-(d+k-2)r^{1-d-k},
\end{cases}
\end{align*}
and
\begin{align*}
\rho_{d}(r_{0})-\rho_{d}(r_{0}+\varepsilon)=&
\begin{cases}
\frac{\varepsilon}{r_{0}}+O(\varepsilon^{2}),&d=2,\\
(d-2)r^{d-3}_{0}\varepsilon+O(\varepsilon^{2}),&d\geq3,
\end{cases}\\
\rho_{4-d-2k}(r_{0}+\varepsilon)-\rho_{4-d-2k}(r_{0})=&(d+2k-2)r_{0}^{d+2k-3}\varepsilon+O(\varepsilon^{2}).
\end{align*}
Then we obtain that for $r\in(r_{0},r_{0}+\varepsilon)$,
\begin{align*}
\partial_{r}\tilde{c}^{_{(i)}}_{k}(r)\sim\varepsilon^{-1},\;\,i=1,2,\,k\geq1,\quad \frac{\tilde{c}^{_{(2)}}_{k}(r)}{r}\sim\varepsilon^{-1}.
\end{align*}
Substituting this into \eqref{QMEIZ001}, we have $|\nabla u|\sim\varepsilon^{-1}$. Recall that $\varepsilon^{-1}$ is the greatest blow-up rate captured in all the above-mentioned work. Then for the perfect conductivity problem with concentric spheres, the gradient of the solution exhibits sharp singularity with respect to the distance $\varepsilon$, which is greatly different from the result in Remark \ref{EMA001} in terms of the insulated conductivity problem.

\end{remark}

The detailed proofs of Theorems \ref{thm1} and \ref{thm2} are given in the next section.

\section{The proof of Theorems \ref{thm1} and \ref{thm2}}
In the following, we will construct the solutions having the form of separation of variables to find the exact solutions for the insulated and perfect conductivity problems. This idea is based on the observation that the considered domain satisfies the radial symmetry.

\begin{proof}[Proof of Theorem \ref{thm1}]
To begin with, for any integer $m\geq0$, consider the following boundary data problem
\begin{equation}\label{AMZW003}
\begin{cases}
\Delta{u}_{m}=0,& \mbox{in}~B_{r_{0}+\varepsilon}\setminus\overline{B}_{r_{0}},\\
\frac{\partial u_{m}}{\partial\nu}=0,&\mbox{on}~\overline{B}_{r_{0}},\\
u_{m}=\varphi_{m},&\mbox{on}~\partial{B}_{r_{0}+\varepsilon},
\end{cases}
\end{equation}
where $\varphi_{m}$ is given by \eqref{AMZW001}. In view of \eqref{EXPAN005}, it follows from the Hopf's Lemma and the maximum principle that
\begin{align*}
\sup\limits_{x\in\overline{B_{r_{0}+\varepsilon}\setminus{B_{r_{0}}}}}|u_{m}(x)-u(x)|\leq \sup\limits_{x\in\partial B_{r_{0}+\varepsilon}}|\varphi_{m}(x)-\varphi(x)|\rightarrow0,\quad\mathrm{as}\;m\rightarrow\infty.
\end{align*}
That is, the sequence $\{u_{m}\}$ converges uniformly to the solution $u$ of the original problem \eqref{ZZWW001} in $\overline{B_{r_{0}+\varepsilon}\setminus B_{r_{0}}}$. Then the problem is reduced to solving the explicit solution for problem \eqref{AMZW003} in the following.

According to \eqref{AMZW001}, we carry out the linear decomposition for the solution $u_{m}$ to problem \eqref{AMZW003} as follows:
\begin{align*}
u_{m}=\sum^{m}_{k=0}\sum^{N_{k,d}}_{l=1}a_{kl}v_{kl},\quad\mathrm{in}\;B_{r_{0}+\varepsilon}\setminus\overline{B}_{r_{0}},
\end{align*}
where, for $k=0,1,...,m$ and $l=1,...,N_{k,d}$, $v_{kl}$ verifies
\begin{align}\label{MZW001}
\begin{cases}
\Delta v_{kl}=0,&\mathrm{in}\;B_{r_{0}+\varepsilon}\setminus\overline{B}_{r_{0}},\\
\frac{\partial v_{kl}}{\partial\nu}=0,&\mathrm{on}\;\partial B_{r_{0}},\\
v_{kl}=Y_{k,l}(\xi),&\mathrm{on}\;\partial B_{r_{0}+\varepsilon}.
\end{cases}
\end{align}
First, in the case of $k=0$, we know that $\frac{\partial v_{00}}{\partial\nu}=0$ on $\partial B_{r_{0}}$ and $v_{00}=Y_{0,0}(\xi)=|\mathbb{S}^{d-1}|^{-1/2}$ on $\partial B_{r_{0}+\varepsilon}$. Then applying the maximum principle and the Hopf's Lemma, we obtain that $v_{00}=|\mathbb{S}^{d-1}|^{-1/2}$ in $B_{r_{0}+\varepsilon}\setminus\overline{B}_{r_{0}}$.

We now proceed to consider the case when $k>0$. Due to the radial symmetry of the domain, we assume that the solution of equation \eqref{MZW001} possesses the form of separation of variables $v_{kl}=f_{k}(r)Y_{k,l}(\xi)$. Recall that under the polar coordinates, the Laplace operator $\Delta$ can be expressed as
\begin{align}\label{ZMZMZ001}
\Delta=\partial_{rr}+\frac{d-1}{r}\partial_{r}+\frac{1}{r^{2}}\Delta_{\mathbb{S}^{d-1}},
\end{align}
where $\Delta_{\mathbb{S}^{d-1}}$ is the Laplace-Beltrami operator satisfying that $\Delta_{\mathbb{S}^{d-1}}Y_{k,l}(\xi)=-k(k+d-2)Y_{k,l}(\xi)$, which implies that spherical harmonics are eigenfunctions of the Laplace-Beltrami operator. Substituting these relations into \eqref{MZW001}, we deduce
\begin{align}\label{QAE001}
\begin{cases}
\partial_{rr}f_{k}(r)+\frac{d-1}{r}\partial_{r}f_{k}(r)-\frac{k(k+d-2)}{r^{2}}f_{k}(r)=0,&\mathrm{for}\;r\in(r_{0},r_{0}+\varepsilon),\\
\partial_{r}f_{k}(r_{0})=0,\\
f_{k}(r_{0}+\varepsilon)=1.
\end{cases}
\end{align}
Picking test function $f_{k}(r)=r^{s}$ and substituting it into the first line of \eqref{QAE001}, we obtain
\begin{align*}
s^{2}+s(d-2)-k(k+d-2)=0,
\end{align*}
which reads that $s=k$ or $s=-(d-2+k)$. Since second-order homogeneous linear differential equation only has two linear independent solutions, then the general solution of \eqref{QAE001} can be written as
\begin{align*}
f_{k}(r)=c_{k}^{_{(1)}}r^{-(d-2+k)}+c_{k}^{_{(2)}}r^{k},
\end{align*}
for some constants $c_{k}^{_{(i)}}$, $i=1,2.$ This, together with the second and third lines of \eqref{QAE001}, shows that
\begin{align*}
\begin{cases}
-c_{k}^{_{(1)}}(d+k-2)r_{0}^{-(d+k-1)}+c_{k}^{_{(2)}}kr_{0}^{k-1}=0,\\
c_{k}^{_{(1)}}r_{0}^{-(d+k-2)}+c_{k}^{_{(2)}}(r_{0}+\varepsilon)^{k}=1.
\end{cases}
\end{align*}
A direct computation shows that the values of $c_{k}^{_{(i)}}$, $i=1,2$ are given by \eqref{COEFF001}. Combining these facts above, we have
\begin{align*}
u_{m}=a_{00}|\mathbb{S}^{d-1}|^{-1/2}+\sum^{m}_{k=1}\sum^{N_{k,d}}_{l=1}a_{kl}(c^{_{(1)}}_{k}r^{2-d-k}+c^{_{(2)}}_{k}r^{k})Y_{k,l}(\xi).
\end{align*}
By sending $m\rightarrow\infty$, the proof of Theorem \ref{thm1} is finished.

\end{proof}

\begin{proof}[Proof of Theorem \ref{thm2}]
Analogously as before, for any integer $m\geq0$, consider
\begin{equation}\label{YHU0.009}
\begin{cases}
\Delta u_{m}=0,& \mbox{in}~B_{r_{0}+\varepsilon}\setminus\overline{B}_{r_{0}},\\
u=C^{0}_{m},&\mbox{on}~\overline{B}_{r_{0}},\\
\int_{\partial{B}_{r_{0}}}\frac{\partial u_{m}}{\partial\nu}=0,\\
u_{m}=\varphi_{m},&\mbox{on}~\partial{B}_{r_{0}+\varepsilon},
\end{cases}
\end{equation}
where the free constant $C^{0}_{m}$ is determined by the third line of \eqref{YHU0.009}, $\varphi_{m}$ is given by \eqref{AMZW001} satisfying \eqref{EXPAN005}. We first split the solution $u_{m}$ of problem \eqref{YHU0.009} as follows:
\begin{align}\label{DECO001}
u_{m}=C^{0}_{m}v_{1}+\sum^{m}_{k=0}\sum^{N_{k,d}}_{l=1}a_{kl}v_{kl},
\end{align}
where $v_{1}$ and $v_{kl}$, respectively, solve
\begin{align}\label{WZM001}
\begin{cases}
\Delta v_{1}=0,&\mathrm{in}\;B_{r_{0}+\varepsilon}\setminus\overline{B}_{r_{0}},\\
v_{1}=1,&\mathrm{on}\;\partial B_{r_{0}},\\
v_{1}=0,&\mathrm{on}\;\partial B_{r_{0}+\varepsilon},
\end{cases}\quad
\begin{cases}
\Delta v_{kl}=0,&\mathrm{in}\;B_{r_{0}+\varepsilon}\setminus\overline{B}_{r_{0}},\\
v_{kl}=0,&\mathrm{on}\;\partial B_{r_{0}},\\
v_{kl}=Y_{k,l}(\xi),&\mathrm{on}\;\partial B_{r_{0}+\varepsilon}.
\end{cases}
\end{align}
In view of the radial symmetry of the domain, we consider the solutions of the following forms:
\begin{align}\label{DZQ001}
v_{1}(x)=g(r),\quad v_{kl}(x)=f_{k}(r)Y_{k,l}(\xi),\quad x=(r,\xi)\in(r_{0},r_{0}+\varepsilon)\times\mathbb{S}^{d-1}.
\end{align}
In light of \eqref{ZMZMZ001} and substituting \eqref{DZQ001} into \eqref{WZM001}, we derive
\begin{align}\label{QZW001}
\begin{cases}
\partial_{rr}g(r)+\frac{d-1}{r}\partial_{r}g(r)=0,&\mathrm{for}\;r\in(r_{0},r_{0}+\varepsilon),\\
g(r_{0})=1,\\
g(r_{0}+\varepsilon)=0,
\end{cases}
\end{align}
and
\begin{align}\label{QZW002}
\begin{cases}
\partial_{rr}f_{k}(r)+\frac{d-1}{r}\partial_{r}f_{k}(r)-\frac{k(k+d-2)}{r^{2}}f_{k}(r)=0,&\mathrm{for}\;r\in(r_{0},r_{0}+\varepsilon),\\
f_{k}(r_{0})=0,\\
f_{k}(r_{0}+\varepsilon)=1.
\end{cases}
\end{align}
From the first line of \eqref{QZW001}, we know that $g(r)$ is the radial solution of the Laplace equation. Then we have
\begin{align*}
v_{1}=
\begin{cases}
c_{2}^{_{(1)}}\ln r+c_{2}^{_{(2)}},&d=2,\\
c_{d}^{_{(1)}}r^{2-d}+c_{d}^{_{(2)}},&d\geq3,
\end{cases}
\end{align*}
for some constants $c_{d}^{_{(i)}}$, $i=1,2,\,d\geq2$. Using the Dirichlet-type boundary value conditions in the second and third lines of \eqref{QZW001}, we solve the explicit values of these constants as follows:
\begin{align*}
\begin{cases}
c_{2}^{_{(1)}}=-\frac{1}{\ln(r_{0}+\varepsilon)-\ln r_{0}},\\
c_{2}^{_{(2)}}=\frac{\ln(r_{0}+\varepsilon)}{\ln(r_{0}+\varepsilon)-\ln r_{0}},
\end{cases}\quad
\begin{cases}
c_{d}^{_{(1)}}=\frac{1}{r_{0}^{2-d}-(r_{0}+\varepsilon)^{2-d}},\\
c_{d}^{_{(2)}}=-\frac{(r_{0}+\varepsilon)^{2-d}}{r_{0}^{2-d}-(r_{0}+\varepsilon)^{2-d}},
\end{cases}\quad\mathrm{for}\;d\geq3.
\end{align*}
By \eqref{NOTA001}, the explicit solution of $v_{1}$ can be unified as
\begin{align}\label{UAQ001}
v_{1}=\frac{\rho_{d}(r)-\rho_{d}(r_{0}+\varepsilon)}{\rho_{d}(r_{0})-\rho_{d}(r_{0}+\varepsilon)},\quad\mathrm{for}\;d\geq2.
\end{align}

With regard to $f_{k}(r)$ defined by \eqref{QZW002}, if $k=0$, then $f_{0}(r)$ becomes the radial solution of the Laplace equation. Then it follows from a similar calculation that
\begin{align}\label{UAQ002}
v_{00}=f_{0}(r)Y_{0,0}(\xi)=\frac{(\rho_{d}(r_{0})-\rho_{d}(r))|\mathbb{S}^{d-1}|^{-1/2}}{\rho_{d}(r_{0})-\rho_{d}(r_{0}+\varepsilon)},\quad\mathrm{for}\;d\geq2.
\end{align}
On the other hand, if $k>0$, similar to \eqref{QAE001}, we also have
\begin{align*}
f_{k}(r)=c_{kl}^{_{(1)}}r^{-(d-2+k)}+c_{kl}^{_{(2)}}r^{k},
\end{align*}
for some constants $c_{k}^{_{(i)}}$, $i=1,2.$ By the second and third lines in \eqref{QZW002}, these constants must verify
\begin{align*}
\begin{cases}
c_{k}^{_{(1)}}r_{0}^{-(d+k-2)}+c_{k}^{_{(2)}}r_{0}^{k}=0,\\
c_{k}^{_{(1)}}(r_{0}+\varepsilon)^{-(d+k-2)}+c_{k}^{_{(2)}}(r_{0}+\varepsilon)^{k}=1,
\end{cases}
\end{align*}
which yields that
\begin{align*}
c_{k}^{_{(1)}}=-\frac{r_{0}^{d+2k-2}(r_{0}+\varepsilon)^{d+k-2}}{(r_{0}+\varepsilon)^{d+2k-2}-r_{0}^{d+2k-2}},\quad c_{k}^{_{(2)}}=\frac{(r_{0}+\varepsilon)^{d+k-2}}{(r_{0}+\varepsilon)^{d+2k-2}-r_{0}^{d+2k-2}}.
\end{align*}
Then we have
\begin{align}\label{UAQ003}
v_{kl}=\frac{(\rho_{2-k}(r)-\rho_{4-d-2k}(r_{0})\rho_{d+k}(r))\rho_{4-d-k}(r_{0}+\varepsilon)}{\rho_{4-d-2k}(r_{0}+\varepsilon)-\rho_{4-d-2k}(r_{0})}Y_{k,l}(\xi).
\end{align}

From \eqref{DECO001}, it remains to calculate the free constant $C^{0}_{m}$. Using the third line of \eqref{YHU0.009}, it follows from integration by parts that
\begin{align}\label{EQU002}
C_{m}^{0}=-\frac{\sum^{m}_{k=0}\sum^{N_{k,d}}_{l=1}a_{kl}\int_{\partial B_{r_{0}}}\frac{\partial v_{kl}}{\partial\nu}}{\int_{\partial B_{r_{0}}}\frac{\partial v_{1}}{\partial\nu}}=\frac{\int_{\partial B_{r_{0}+\varepsilon}}\frac{\partial v_{1}}{\partial\nu}\varphi_{m}}{\int_{\partial B_{r_{0}}}\frac{\partial v_{1}}{\partial\nu}}.
\end{align}
By the same argument, we have
\begin{align*}
C^{0}=\frac{\int_{\partial B_{r_{0}+\varepsilon}}\frac{\partial v_{1}}{\partial\nu}\varphi}{\int_{\partial B_{r_{0}}}\frac{\partial v_{1}}{\partial\nu}}.
\end{align*}
Using \eqref{EXPAN005}, we obtain that $C_{m}^{0}\rightarrow C^{0}$, as $m\rightarrow\infty$. For simplicity, denote
\begin{align*}
\chi_{_{d}}=
\begin{cases}
1,&d=2,\\
d-2,&d\geq3.
\end{cases}
\end{align*}
By a straightforward computation, we obtain
\begin{align*}
\int_{\partial B_{r_{0}}}\frac{\partial v_{1}}{\partial\nu}=-\frac{|\mathbb{S}^{d-1}|\chi_{d}}{\rho_{d}(r_{0})-\rho_{d}(r_{0}+\varepsilon)},
\end{align*}
and
\begin{align*}
\int_{\partial B_{r_{0}+\varepsilon}}\frac{\partial v_{1}}{\partial\nu}\varphi_{m}=-\frac{\chi_{d}}{\rho_{d}(r_{0})-\rho_{d}(r_{0}+\varepsilon)}\sum^{m}_{k=0}\sum^{N_{k,d}}_{l=1}a_{kl}\fint_{\mathbb{S}^{d-1}}Y_{k,l}(\xi)d\xi.
\end{align*}
Therefore, substituting these two equations into \eqref{EQU002}, we obtain
\begin{align*}
C_{m}^{0}=&\sum^{m}_{k=0}\sum^{N_{k,d}}_{l=1}a_{kl}\fint_{\mathbb{S}^{d-1}}Y_{k,l}(\xi)d\xi.
\end{align*}
This, together with \eqref{DECO001} and \eqref{UAQ001}--\eqref{UAQ003}, shows that
\begin{align}\label{WMDZQ001}
u_{m}=\frac{a_{00}}{|\mathbb{S}^{d-1}|^{1/2}}+\sum^{m}_{k=1}\sum^{N_{k,d}}_{l=1}a_{kl}\Big(\tilde{c}^{_{(1)}}_{k}(r)\fint_{\mathbb{S}^{d-1}}Y_{k,l}(\xi)d\xi+\tilde{c}^{_{(2)}}_{k}(r)Y_{k,l}(\xi)\Big),
\end{align}
where $\tilde{c}^{_{(i)}}_{k}(r)$, $i=1,2,\,k=1,...,m$ are defined in \eqref{CONSTANT006}. Combining \eqref{EXPAN005} and the maximum principle, we have
\begin{align*}
\sup\limits_{x\in\overline{B_{r_{0}+\varepsilon}\setminus{B_{r_{0}}}}}|u_{m}(x)-u(x)|\leq \max\Big\{\sup\limits_{x\in\partial B_{r_{0}+\varepsilon}}|\varphi_{m}(x)-\varphi(x)|,|C_{m}^{0}-C^{0}|\Big\}\rightarrow0,
\end{align*}
as $m\rightarrow\infty$. Hence, by letting $m\rightarrow\infty$ in \eqref{WMDZQ001}, we complete the proof.

\end{proof}

\noindent{\bf{\large Acknowledgements.}}
The author would like to thank Prof. C.X. Miao for his constant encouragement and useful discussions. The author was partially supported by CPSF (2021M700358).


\bibliographystyle{plain}

\def\cprime{$'$}

\end{document}